\documentclass[11pt]{article}
\usepackage{amsmath,amsthm,amsfonts,amsthm,amssymb, mathtools,longtable,multirow,subfig}
\usepackage{mathrsfs,array,graphicx,color,latexsym,tikz,calc,pifont,natbib}
\usetikzlibrary{shadows,patterns,arrows,decorations.pathreplacing}
\usetikzlibrary{calc,arrows.meta,arrows,decorations.markings,chains,positioning}

\setcitestyle{numbers,square}

\DeclareMathOperator{\dv}{div}
\DeclareMathOperator{\diag}{diag}
\DeclareMathOperator{\curl}{curl}
\DeclareMathOperator{\grad}{grad}
\DeclareMathOperator{\Sp}{Sp}

\theoremstyle{plain}
\newtheorem{thm}{\bf Theorem}[section]
\newtheorem{cor}[thm]{\bf Corollary}
\newtheorem{df}[thm]{\bf Definition}
\newtheorem{re}[thm]{\bf Remark}
\newtheorem{lem}[thm]{\bf Lemma}
\newtheorem{prop}[thm]{\bf Proposition}
\newtheorem{problem}[thm]{\bf Problem}
\newtheorem{examplex}[thm]{\bf Example}

\renewenvironment{abstract}{\par\noindent\textbf{\abstractname.}\ \ignorespaces}{\par\medskip}
\providecommand{\keywords}[1]{\small \textbf{\textit{Keywords.}} #1}
\providecommand{\ams}[1]{\small \textbf{\textit{AMS subject classifications.}} #1}

\title{\bf \Large  Helmholzian spectra of graphs: basic properties\vspace{-0.5em}}

\author{
{\small   Lu Lu$^a$, \ \ Yongtang Shi$^b$, \ \ Zoran Stani\'c$^{c}$, \ \ Jianfeng Wang$^{d,}$\footnote{Corresponding author.\newline{Email addresses:} lulugdmath@163.com (L. Lu), shi@nankai.edu.cn (Y.T. Shi), zstanic@matf.bg.ac.rs (Z. Stani\' c), jfwang@aliyun.com (J.F. Wang), wangy@ahu.edu.cn (Y. Wang).}\;, \ \ Yi Wang$^e$ }\\[2mm]
\footnotesize $^a$School of Mathematics and Statistics, HNP-LAMA, Central South University,
Changsha 410083, China\\
\footnotesize $^b$Center for Combinatorics and LPMC, Nankai University, Tianjin 300071, China\\
\footnotesize $^c$Faculty of Mathematics, University of Belgrade, Studentski trg 16, Belgrade, Serbia\\
\footnotesize $^d$School of Mathematics and Statistics, Shandong University of Technology, Zibo 255049, China\\
\footnotesize $^e$School of Mathematics, Anhui University, Hefei 230039, China\\}

\date{}

\begin{document}
\maketitle

\begin{abstract}
The Helmholtzian matrix of a graph $G=(V(G),E(G))$ is a graph-theoretic analogue of the vector Laplacian (or Helmholtz operator) [S. Li, L. Lu, J.F. Wang, A graph discretization of vector Laplacian,  379 (2026) 446--460]. Motivated by the applications of graph Helmholtzian in simplicial networks, we will investiagte its basic spectral properties. As the first graph matrix indexed by edge set, we find that Helmholtzian matrix is positive semi-definite and its non-negativity correlates with the odd cycles in $G$ and the orientation on $E(G)$, while its irreducibility relates to the signed graphs with loops. We show that the eigenvalues of Helmholtzian matrix are independent of the orientation and further investigate the eigenvalue interlacing under edge additions. One of striking findings is that the non-zero  eigenvalues of the Laplacian matrix are those of Helmholtzian matrix of every graph.  All these discoveries reveal that the Helmholtzian spectrum of $G$ balances and bridges the oriented graphs, weighted graphs and signed graphs as well as their adjacency or Laplacian spectra.
\end{abstract}

{\noindent}\keywords Graph spectrum; Helmholtzian matrix; Laplacian matrix; Signed graph; Oriented graph. \\[-4mm]

{\noindent}\ams 05C50, 05C22.


\section{Introduction}

Let $G = (V(G),E(G))$ be a simple and undirected graph.  The \textit{adjacency matrix} of $G$, denoted by $A(G) = (a_{ij})$, has rows and columns indexed by the vertices, along with $a_{ij} =1$ if $v_iv_j \in E(G)$ and $a_{ij} =0$ otherwise. Then the \textit{Laplacian matrix} of $G$ is defined as $L(G) = D(G)-A(G)$ where $D(G) = \diag(d(v_1),d(v_2),\ldots,d(v_n))$ is the diagonal matrix of vertex degrees in $G$. As known, the Laplacian matrix  $L(G)$ is a graph discretization of scalar Laplacian $\nabla^2 = -\dv\grad$, where $\dv$ and $\grad$ are respectively the divergence and gradient operators in the real space $\mathbb{R}^3$ \cite[eg.]{Hodge-lim}. Recently, the authors \cite{li-lu-wang} gave the graph matrix presentation of vector Laplacian
\begin{equation}\label{vector-F1}
-\nabla^2 \mathbf{F}= \grad\grad^*\mathbf{F}+\curl^*\curl \mathbf{F},
\end{equation}
where the operator $\curl$ is over a vector field  $\mathbf{F}$ in $\mathbb{R}^3$ as well as $\grad^*$ and $\curl^*$ are respectively the adjoint operators. In order to describe this matrix, we introduce some necessary terminologies and notations from \cite{li-lu-wang}.

For an ordinary graph $G$,  we will simultaneously deal with edge orientations and associated graphs having loops. In further considerations, we will also associate graphs to particular signed graphs. In short, the main subject are ordinary graphs, and their generalizations are considered as tools. For the any orientation on the edges and triangles of $G$, let $e^+$ and $e^-$ be the {\it head} and the {\it tail} of an oriented edge $e$ respectively. Set $\mathcal{V}(e)=\{e^+, e^-\}$. Write $u\rightarrow v$ if there is an oriented edge from $u$ to $v$. Set $u\sim v$ if $u$ and $v$ are adjacent, and $u \nsim v$ otherwise. Thus, $u\sim v$ indicates either $u\rightarrow v$ or $v\rightarrow u$. Denote by $v \in e$ if a vertex $v\in \mathcal{V}(e)$. Moreover, let $e\rightarrow v$ if $v=e^+$ and $v\rightarrow e$ if $v=e^-$. For two edges $e_1$ and $e_2$, put $e_1\sim e_2$ if
$\mathcal{V}(e_1)\cap \mathcal{V}(e_2)\ne\emptyset$. Denote by $e_1\rightarrow e_2$ if $e_1^+=e_2^-$, $e_1\overset{+}{\sim}e_2$ if $e_1^+=e_2^+$, and
$e_1\overset{-}{\sim}e_2$   if $e_1^-=e_2^-$. Let $e_1\leftrightarrow e_2$ mean that
either $e_1\rightarrow e_2$ or $e_2 \rightarrow e_1$,   $e_1\overset{\pm}{\sim}e_2$ mean that
either $e_1\overset{+}{\sim}e_2$ or $e_1\overset{-}{\sim}e_2$, and  $e_1\vartriangle e_2$ mean that $e_1$
and $e_2$ are in the same triangle. For an edge $e$ and a triangle
$\vartriangle$, set $e\in\vartriangle$ if $e$ is an edge of $\vartriangle$. Moreover, if
the orientation of $e$ agrees with that of $\vartriangle$, then set
$e\in\vartriangle^+$, and $e\in\vartriangle^-$ otherwise. For clarity, the symbols are summarized in Table 1, as they will be used frequently throughout this paper.

\begin{table}[h]
		\centering
		\begin{minipage}{0.48\textwidth}
			\centering
			\begin{tabular}{|c|c|c|c|}
				\hline
				\multicolumn{3}{|c|}{Symbol} & Diagram \\
				\hline
				\multirow{2}*{$u\sim v$} & \multicolumn{2}{c|}{$u\rightarrow v$} &
				\includegraphics[scale=0.6]{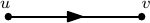}  \\[1.6mm]
				\cline{2-4}
				~ & \multicolumn{2}{c|}{$v\rightarrow u$} & \includegraphics[scale=0.6]{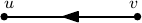}  \\[1.6mm]
				\hline
				\multirow{2}*{$u\in e$} & \multicolumn{2}{c|}{$u\rightarrow e$~($u=e^-$)} &
				\includegraphics[scale=0.6]{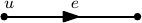}  \\[1.6mm]
				\cline{2-4}
				~ & \multicolumn{2}{c|}{$e\rightarrow u$~($u=e^+$)} & \includegraphics[scale=0.6]{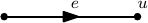}  \\[1.6mm]
				\hline
				\multirow{4}*{$e_1\sim e_2$} & \multirow{2}*{$e_1\overset{\pm}{\sim}e_2$} &
				$e_1\overset{+}{\sim}e_2$ & \includegraphics[scale=0.6]{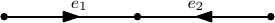}  \\[1.6mm]
				\cline{3-4}
				~ & ~ & $e_1\overset{-}{\sim}e_2$ & \includegraphics[scale=0.6]{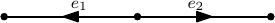} \\[1.6mm]
				\cline{2-4}
				~& \multirow{2}*{$e_1\leftrightarrow e_2$} & $e_1\rightarrow e_2$ &
				\includegraphics[scale=0.6]{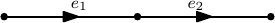}\\[1.6mm]
				\cline{3-4}
				~ & ~ & $e_2\rightarrow e_1$ &  \includegraphics[scale=0.6]{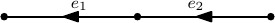}\\[1.6mm]
				\hline
			\end{tabular}
		\end{minipage}
		\hfill
		\begin{minipage}{0.48\textwidth}
			\centering
			\begin{tabular}{|c|c|c|c|}
				\hline
				\multicolumn{3}{|c|}{Symbol} & Diagram \\
				\hline
				\multicolumn{3}{|c|}{$e_1\vartriangle e_2$} &
				\raisebox{-.5\height}{\includegraphics[scale=0.65]{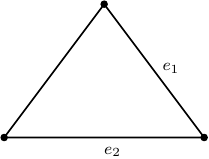}}\\[1.6mm]
				\hline
				\multirow{2}*{$e\in\vartriangle$} & \multicolumn{2}{c|}{$e\in\vartriangle^+$} &
				\raisebox{-.5\height}{\includegraphics[scale=0.75]{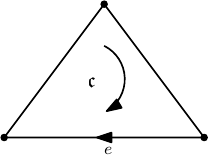}}\\[1.6mm]
				\cline{2-4}
				~ & \multicolumn{2}{c|}{$e\in\vartriangle^-$} &
				\raisebox{-.5\height}{\includegraphics[scale=0.75]{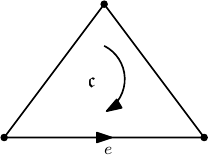}} \\[1.6mm]
				\hline
			\end{tabular}
		\end{minipage}
\caption{Relations between vertices, edges and triangles.}
	\end{table}

For each vertex $u \in V(G)$, the {\it in-neighbourhood}, the {\it out-neighbourhood} and
the {\it neighbourhood}  are  defined as
$$N_G^+(u)=\{v\in V\mid v\rightarrow u\},\, N_G^-(u)=\{v\in V\mid u\rightarrow v\}~
\text{and}~  N_G(u)=N_G^+(u)\cup N_G^-(u),$$ respectively. The cardinalities $d_G^+(u)=|N_G^+(u)|$,
$d_G^-(u)=|N_G^-(u)|$ and $d_G(u)=|N_G(u)|$ are called the {\it in-degree}, the {\it
out-degree} and the {\it degree} of $u$. For an edge $e \in E(G)$, let
$\triangle_G(e)=|\{\vartriangle\in T(G)\mid e\in \vartriangle\}|$ be the {\it triangle degree of $e$}, which is also denoted by $\triangle(e)$ when $G$ is unambiguous. As usual, we
write $O_{x\times y}$, $J_{x\times y}$, $I_{x}$, $\mathbf{0}_x$ $\mathbf{1}_x$ and  for the all-zero
matrix, the all-one matrix, the identity matrix, the all-0 column vector and the all-1 column vector of the corresponding size, respectively.
The subscript may be omitted. The {\it edge-vertex incidence matrix}
$\mathcal{B}=\mathcal{B}(G)=(b_{ev})_{m\times n}$ and the {\it triangle-edge incidence matrix}
$\mathcal{C}=\mathcal{C}(G)=(c_{\vartriangle e})_{t\times m}$ are defined as
\begin{equation*}\label{B-C}
b_{ev}=
\begin{cases}
-1,& v\rightarrow e\\
1,&e\rightarrow v\\
0,&\text{otherwise}
\end{cases} \;\;\text{and}\;\; c_{\vartriangle e}=
\begin{cases}-1,&
e\in\vartriangle^-\\
1,&e\in\vartriangle^+\\
0,&\text{otherwise}.
\end{cases}
\end{equation*}
For example, considering the graph $G$ of Fig.~\ref{fig-1}  we obtain

\begin{figure}[htbp]
\begin{minipage}{0.4\linewidth}
\includegraphics[scale=0.7]{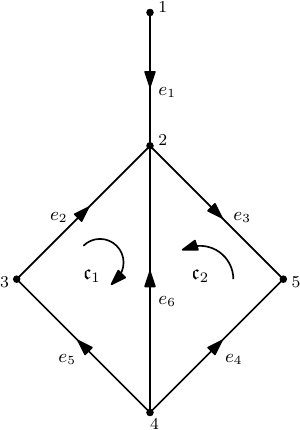}
\end{minipage}
\hspace{-2.5cm}
\begin{minipage}{0.6\linewidth}
{\footnotesize\[\mathcal{B}(G)=\left(\begin{array}{ccccc}
-1&1&0&0&0\\
0&1&-1&0&0\\
0&-1&0&0&1\\
0&0&0&-1&1\\
0&0&1&-1&0\\
0&1&0&-1&0
\end{array}\right)~~ \text{and}~~
\mathcal{C}(G)=\left(\begin{array}{cccccccc}
0&1&0&0&1&-1\\
0&0&-1&1&0&-1
\end{array}\right).\]}
\end{minipage}
\caption{The matrices $\mathcal{B}(G)$ and $\mathcal{C}(G)$ of the presented graph $G$.}
\label{fig-1}
\end{figure}

We are now in the position to give the graph matrix representation of vector Laplacian, named by the {\it Helmholtzian matrix} $\mathcal{H}(G)$ of a graph $G$ \cite{li-lu-wang}.

\begin{df}{\rm \cite{li-lu-wang}}\label{H-def}
Let $G$ be a graph  equipped with orientations on its edge set $E$ and triangle set $T$. Then the Helmholtzian matrix $\mathcal{H}(G)=(h_{ee'})$  of $G$ is defined as follows:
$$
h_{ee'}=
\begin{cases}
\triangle(e)+2, & \mbox{if $e'=e$}\\
-1, & \mbox{if $e\leftrightarrow e'$ and $e\not\vartriangle e'$}\\
1, & \mbox{if $e'\overset{\pm}{\sim} e$ and $e\not\vartriangle e'$} \\
0, & \mbox{otherwise}.
\end{cases}
$$
\end{df}

Previously, the Helmholtzian matrix is denoted as the sum ${\mathcal B}{\mathcal B}^{\intercal}+{\mathcal C}^{\intercal}{\mathcal C}$ \cite{Hodge-lim}, which has been applied in statistical ranking \cite{jiang-lim} and simplicial networks \cite{sch-Siam}. However, it has attracted a very sporadic attention on its basic spectral properties, which will be investigated in this paper.

In contrast to many other matrices associated with graphs, it is indexed by the edge set. By a convention, an edgeless graph is represented by  the $0\times 0$ matrix. Throughout the entire paper, unless told otherwise, we assume that a graph under consideration has at least one edge!

\begin{examplex}\label{ex-0} We return, once again, to the graph $G$ illustrated in Fig.~\ref{fig-1}.
From Theorem \ref{H-def}, one can easily deduce  that
\[\mathcal{H}(G)=\begin{pmatrix}
2&1&-1&0&0&1\\
1&3&-1&0&0&0\\
-1&-1&3&0&0&0\\
0&0&0&3&1&0\\
0&0&0&1&3&0\\
1&0&0&0&0&4
\end{pmatrix},\]
which is just $\mathcal{B}\mathcal{B}^{\intercal} +\mathcal{C}^{\intercal}\mathcal{C}$.
\end{examplex}

In the remainder of this paper, we will investigate the basic spectral properties of the graph Helmholtzian, along with some additional topics. Precisely, in Section \ref{H40} we show that the eigenvalues of Helmholtzian matrices (or $\mathcal{H}$-eigenvalues) are independent of the orientations of graphs.  In Section~\ref{H43} we consider the irreducibility of this matrix and relate the corresponding graph to a particular signed graph. In Section~\ref{H42} we deal with the interlacing between the $\mathcal{H}$-eigenvalues of a graph.  Additional relations between $\mathcal{H}$-eigenvalues and Laplacian eigenvalues are established in Subsection~\ref{H41}. In particular, we prove that the $\mathcal{H}$-spectrum of a graph is closely related to the Laplacian spectrum of the same graph and the spectrum of an associated graph.

\section{Orientation invariance of Helmholtzian eigenvalues}\label{H40}

It is worth mentioning that  the Helmholtzian matrix
$\mathcal{H}(G)$ does not depend on the orientations of  triangles, even
though the matrix $\mathcal{C}$ does depend on these orientations. Thus, it suffices to
consider graphs with  orientation assigned to their edge set. Although the matrix $\mathcal{H}(G)$
leans on the orientation of edges, the next result establishes that the eigenvalues of
$\mathcal{H}(G)$ do not depend on the choice of orientation.

A \textit{switching matrix} is the diagonal matrix of $\pm1$s. Two matrices $X$ and $Y$ are {\it switching similar} if there is a switching
matrix $S$ such that $Y=S^{-1}XS = SXS$.

\begin{thm}\label{lem-2}
Let $G$ be an oriented graph and $G'$ the oriented graph obtained from $G$ by switching the direction of an
edge $e$. Then $\mathcal{H}(G)$ and $\mathcal{H}(G')$ are switching similar, and thus they share the same eigenvalues.
\end{thm}

\begin{proof}
Note that the relations between any pair of edges in $E(G)\setminus \{e\}$ are retained in $G'$. This means that $$\mathcal{H}(G)_{e_x,e_y}=\mathcal{H}(G')_{e_x,e_y}$$ holds for every pair $e_x,e_y\in E(G)\setminus\{e\}$. Also, one can easily verify that $\mathcal{H}(G')_{e,e'}=0$ holds whenever $\mathcal{H}(G)_{e,e'}=0$. If $\mathcal{H}(G)_{e,e'}=1$, then $e\overset{\pm}{\sim} e'$ and $e\not\vartriangle e'$ in $G$. Hence, $e\leftrightarrow e'$ and $e\not\vartriangle e'$ in $G'$ which implies $\mathcal{H}(G')_{e,e'}=-1$. Similarly, we have $\mathcal{H}(G')_{e,e'}=1$ whenever $\mathcal{H}(G)_{e,e'}=-1$. Therefore, $$\mathcal{H}(G')=S\mathcal{H}(G)S,$$ where $S$ is the diagonal matrix in which the $e$th diagonal entry is $-1$, whereas the remaining diagonal entries are~$1$. The desired result follows since similar matrices share the same spectrum.
\end{proof}

The {\it Helmholtzian spectrum} (or the \textit{$\mathcal{H}$-spectrum})  $\Sp_{\mathcal{H}}(G)$ of an undirected graph $G$ is the spectrum of the Helmholtzian matrix of an oriented graph obtained by assigning an edge orientation to $G$. Since both  $\mathcal{B}\mathcal{B}^\intercal$ and $\mathcal{C}^\intercal\mathcal{C}$ are positive semi-definite, the same holds for their sum $\mathcal{H}(G)$. Observing  $\mathcal{H}(G)$ is real and symmetric, we deduce that its eigenvalues, called {\it Helmholtzian eigenvalues} (or {\it $\mathcal{H}$-eigenvalues}), are non-negative real numbers. Therefore, they can be listed as
\begin{equation*}\label{H-eigen1}
\lambda_1(G)\ge\lambda_2(G)\ge\cdots\ge\lambda_m(G) \geq 0,
\end{equation*}
where $m$ is the number of edges, of course.
In particular, if $\lambda_1 > \lambda_2 > \cdots > \lambda_k$ are distinct $\mathcal{H}$-eigenvalues of $G$, then the  {\it $\mathcal{H}$-spectrum} of $G$ is written as
$$
{\Sp}_{\mathcal{H}}(G) = \left\{\!\!\begin{array}{cccc}
                                 \lambda_1(G) & \lambda_2(G) & \cdots & \lambda_k(G) \\
                                  m_1  &  m_2     & \cdots &  m_k
                                \end{array}
                         \!\!\right\},
$$
where $m_i$ is multiplicity of $\lambda_i$ ($1 \leq i \leq k$).

To provide a better insight into the $\mathcal{H}$-spectrum, we compare it to the standard Laplacian spectrum and compute it for some graphs. The comparison is given in the following remark.

\begin{re}\label{relation-HL}
	 It follows from definition of the Laplacian matrix of a graph $G$ that $L(G) = \mathcal{B}^{\intercal}\mathcal{B}$, where $\mathcal{B}$ is the foregoing edge-vertex incidence matrix. Now, if
	$G$ is triangle-free, from Definition \ref{H-def} we obtain $\mathcal{H}(G) = \mathcal{B}\mathcal{B}^{\intercal}$. Consequently, the $\mathcal{H}$-spectrum of $G$ is comprised of its
	non-zero Laplacian eigenvalues along with $m-n+1$ zero eigenvalues. However, it is an established fact that the triangles are largely involved in small-world networks, revealing that the $\mathcal{H}$-spectrum would be more helpful in their study~\cite{shi-chen,wat-str}.
\end{re}

We proceed with examples.

\begin{examplex}\label{ex-1}
The $\mathcal{H}$-spectrum of the complete graph $K_n$ is
$$
{\Sp}_{\mathcal{H}}(K_n) = \left\{\!\!\begin{array}{c}
                                 n  \\
                                 {n \choose 2}\\
                                \end{array}
                         \!\!\right\},
$$
because $\mathcal{H}(K_n) = \diag(n,n,\ldots,n)$ for every orientation of $K_n$.
\end{examplex}

\begin{examplex}
The $\mathcal{H}$-spectrum of the path $P_n$ is
$$
{\Sp}_{\mathcal{H}}(P_n) =
\left\{\!\!\begin{array}{ccccc}
2\cos\frac{\pi}{n}+2 &  2\cos\frac{2\pi}{n}+2 & \cdots & 2\cos\frac{(n-1)\pi}{n}+2
\\
1  & 1 & \cdots & 1
\end{array}
\!\!\right\}.
$$
 Namely, if $P_n$ is assigned an orientation such that $e\overset{\pm}{\sim}
e'$ whenever $e\sim e'$, then $\mathcal{H}(P_n)=A(P_{n-1})+2I$. Since  the eigenvalues of $A(P_{n-1})$ are $2\cos\frac{j\pi}{n}$ ($1 \leq j \leq  n-1$) \cite[p.~73]{cve-doob-sachs-book}, we arrive at the desired $\mathcal{H}$-spectrum.
\end{examplex}

On the basis of Remark~\ref{relation-HL}, one may obtain the Laplacian eigenvalues of $P_n$. Indeed, since $P_n$ is triangle-free, we have $\mathcal{H}(P_n)=\mathcal{B}\mathcal{B}^{\intercal}$, which yields that the Laplacian
eigenvalues of $P_n$ are 0 and $2\cos\frac{\pi j}{n}+2$  ($1\le j\le n-1$). More examples computing the $\mathcal{H}$-spectrum are given in the Subsection~\ref{H43}.

Observing that adding isolated vertices does not affect the Helmholtzian matrix, we deduce that a graph is not reconstructible from this matrix. This brings us to the following problem to end this section.

\begin{problem}
	Decide under which assumptions a graph is uniquely reconstructed from its Helmholtzian matrix.
\end{problem}

\section{Irreducibility and a relationship with signed graphs}\label{H43}

A symmetric square matrix $M$ is {\it reducible} if there exists a permutation matrix $P$, such that
$$P^{-1}MP=\begin{pmatrix}
     E&O\\O&F
\end{pmatrix},$$
where $E$ and $F$ are square matrices. Otherwise, $M$ is called {\it irreducible}. Generally speaking, it is more advantageous if a matrix associated with a connected graph is irreducible, due to the
fact that the irreducible matrices have many helpful properties~\cite{cve-doob-sachs-book}. Examining
irreducibility of the Helmholtzian matrix, we give a simple criterion.

\begin{prop}
Let $G$ be a graph with edge set $E$. Then $\mathcal{H}(G)$ is reducible if and only if there is a partition
$E=E_1\cup E_2$, such that for every pair of edges $e_1\in E_1$ and $e_2\in E_2$ either
$e_1\vartriangle e_2$ or $e_1\not\sim e_2$.
\end{prop}

Say, for the graph $G$ of Fig.~\ref{fig-1} we may take  $E_1=\{e_4,e_5\}$ and $E_2=E\setminus
E_1$. Then the partition $E=E_1\cup E_2$ is as in the previous proposition, and thus $\mathcal{H}(G)$ is
reducible.

We shall return to reducibility soon. At this moment we need to introduce signed graphs. A {\it signed graph} $\dot{G}$ is a graph in which every edge is either  positive or negative. Precisely, $\dot{G}$ is an ordered pair $(G,\sigma)$, where $G$ is an ordinary graph, called the underlying graph, and
$\sigma\colon E(G)\longrightarrow\{+1,-1\}$ is the signature defined on the edge set of $G$. The adjacency matrix  $A(\dot{G})$  of $\dot{G}$ is obtained from the adjacency matrix of
its underlying graph  by reversing the sign of all 1s which correspond to negative edges. See \cite{wis-dam} for more details about signed directed graphs.


For a graph $G$ with an
orientation on $E(G)$, let $\Lambda(G)$ denotes the following signed graph with loops:
\begin{enumerate}
\item[(a)] the vertex set is $V=E(G)$;
\item[(b)] there are $\triangle(e)+2$ loops with sign $1$ attached at every vertex $e$;
\item[(c)] the set of edges with sign $1$ is $E^+=\{\{e,e'\}\mid e\overset{\pm}{\sim} e'
    \text{ and }e\not\vartriangle e'\}$;
\item[(d)] the set of edges with sign $-1$ is $E^-=\{\{e,e'\}\mid e\leftrightarrow
    e'\text{ and }e\not\vartriangle e'\}$.
\end{enumerate}

The previous signed graph can be viewed as a combination of the Gallai graph and the standard signed line graph \cite{le}. Apparently, there are several definitions of a signed line graph, and a short discussion on them can be found in \cite{BSZ}. The adjacency matrix $A(\Lambda(G)) = (a_{ij}) $  is given by
$$
a_{ij}=
\begin{cases}
	\triangle(e)+2, & \mbox{if $i=j$}  \\
	1, & \mbox{if $ij \in E(\Lambda(G))$ with sign 1}  \\
	-1, & \mbox{if $ij \in E(\Lambda(G))$ with sign $-1$}\\
	0, & \text{otherwise}.
\end{cases}
$$

By deleting the loops from $\Lambda(G)$, we arrive at its reduction, denoted by $\Lambda_R(G)$. In the light of Theorem~\ref{H-def}, we have \begin{equation}\label{eq:rel}\mathcal{H}(G)=A(\Lambda(G))=A(\Lambda_R(G))+\mathcal{D}(G),\end{equation} where $\mathcal{D}(G)=\diag(\triangle(e)+2\mid e\in E(G))$.

It is well-known that
the $A(\Lambda(G))$ is irreducible if and only if  $\Lambda(G)$ is connected; i.e., if and only if  $\Lambda_R(G)$ is connected. Therefore, the next result follows immediately.

\begin{prop}\label{Hel-bukeyue}
For a graph $G$, the  Helmholtzian $\mathcal{H}(G)$ is irreducible if and only if
$\Lambda_R(G)$ is connected.
\end{prop}

The identities \eqref{eq:rel} enable us to transfer the consideration to the domain of signed graphs, and use some known results on them. Two signed graphs $\dot{G}$ and $\dot{G}'$ are {\it switching equivalent} if there is a
switching matrix $S$ such that $A(\dot{G})=S^{-1}A(\dot{G}')S$. In other words $\mathcal{H}(G)$ and $\mathcal{H}(G')$  are switching similar if and only if $\Lambda_R(G)$ and $\Lambda_R(G')$ are switching equivalent.

We know from \cite{Zas} that  a signed graph switches to its underlying graph if and only if each cycle in it is positive (i.e., the product of its edge signed is positive).
This brings us to the following result.

\begin{thm}\label{prop:I}
If there exists an orientation on $G$ such that $\mathcal{H}(G)$ is non-negative, then $G$ has no induced odd cycle of length $\geq 5$. Conversely, if $G$ has no odd cycle of length $\geq 5$, then there exists an orientation on $G$ such that $\mathcal{H}(G)$ is non-negative.
\end{thm}

\begin{proof}
Observe that every triangle in $\Lambda_R(G)$ arises from a triple of edges sharing the same vertex. It follows by definition that such a triangle is positive. Indeed, if all edges are oriented to the common vertex then all the edges in the triangle are positive, and reversing the orientation of an edge in a triple, changes the sign of exactly two edges in the triangle, which remains it positive.

We claim that, for each induced cycle $C$ with edges $e_1,\ldots,e_k$ where $k\ge 4$ of $G$, the vertices $e_1,\ldots,e_k$ form an induced cycle in $\Lambda_R(G)$. Since $e_i\not\vartriangle e_{i+1}$ and $e_i\sim e_{i+1}$ for $1\le i\le k$ where $e_{k+1}=e_1$, $e_i$ and $e_{i+1}$ are adjacent in $\Lambda_R(G)$ as vertices. Hence, the vertices $e_1,\ldots,e_k$ form a cycle in $\Lambda_R(G)$, denoted by $\tilde{C}$. Moreover, if there is a chord edge $e_ie_j$ with $|i-j|\ge 2$, then $e_i\sim e_j$. Therefore, either $e_i,e_{i+1},e_j$ share a common vertex in $C$ or $e_i,e_{i-1},e_j$ share a common vertex in $C$. This means that there is a vertex in $C$ with degree $\ge 3$ in $C$, which is impossible. Hence, $\tilde{C}$ is an induced cycle of length $k$. Conversely, for each induced cycle $C'$ with vertices $e_1,\ldots,e_k$ where $k\ge 4$ in $\Lambda_R(G)$, the edges $e_1,\ldots,e_k$ forms a cycle in $G$. Since $e_i$ and $e_{i+1}$ are adjacent in $\Lambda_R(G)$, we get $e_i\sim e_{i+1}$ and $e_i\not\vartriangle e_{i+1}$ in $G$. Moreover, there is no edge $e_j\in\{e_1,\ldots,e_{k}\}\setminus \{e_{i-1},e_{i+1}\}$ such that $e_i\sim e_j$. Otherwise, the vertices $e_j,e_i,e_{i-1}$ or $e_j,e_j,e_{i+1}$ would form a triangle in $\Lambda_R(G)$, which is impossible because $C'$ is an induced cycle. Therefore, the edges $e_1,\ldots,e_k$ forms a cycle of length $k$ in $G$.

Assume that there is an orientation on $G$ such that $\mathcal{H}(G)$ is non-negative. Then the graph $\Lambda_R(G)$ is switching equivalent to its underlying graph. Suppose to the contrary that $G$ contains an induced cycle $C_{2k+1}$ for some $k\ge 2$. Let $\tilde{C}_{2k+1}$ be the induced cycle in $\Lambda_R(G)$ corresponding to $C$. Note that, changing the orientation of an edge $e_i$ in $C$ would change the signs of both edges $e_{i-1}e_i$ and $e_ie_{i+1}$ in $\tilde{C}$, which remains the sign of $\tilde{C}$. This indicates that the orientation of $C$ does not affect the signature of $\tilde{C}$. Without loss of generality, assume that the orientation of $C$ is clockwise. Therefore, the sign of $\tilde{C}$ is negative, a contradiction.

Conversely, assume that $G$ contains no odd cycle of length $\ge 5$. It suffices to show that $\Lambda_R(G)$ is switching equivalent to its underlying graph, that is, all induced cycles of $\Lambda_R(G)$ is positive. Suppose to the contrary that $\Lambda_R(G)$ contains a negative induced cycle $C'$. Clearly, the length of $C'$ is greater than $3$ because all $3$-cycles in $\Lambda_R(G)$ are positive. Let $C_0$ be the cycle (maybe not induced cycle) in $G$ corresponding to $C'$. Since the orientation of $C_0$ does not affect the signature of $C'$, by considering the clockwise orientation, we conclude that $C'$ is an odd cycle and so is $C_0$. This yields the contradiction.

The proof is completed.
\end{proof}

\begin{examplex}\label{exa-3}
	Observe that $\Lambda(G)$ and $\Lambda_R(G)$ are essentially different objects, as $\Lambda_R(K_3)\cong \Lambda_R(3K_2)$ but $\Lambda(K_3)\ncong \Lambda(3K_2)$.
	
Let $G_1$ and $G_2$ be oriented graphs illustrated in Fig. 2. They are obtained by assigning different orientations to the same graph. Then the signed graphs $\Lambda_R(G_1)$ and $\Lambda_R(G_2)$ are depicted in the same figure, where dotted lines are negative.  By Proposition~\ref{prop:I} they are switching equivalent.
\end{examplex}

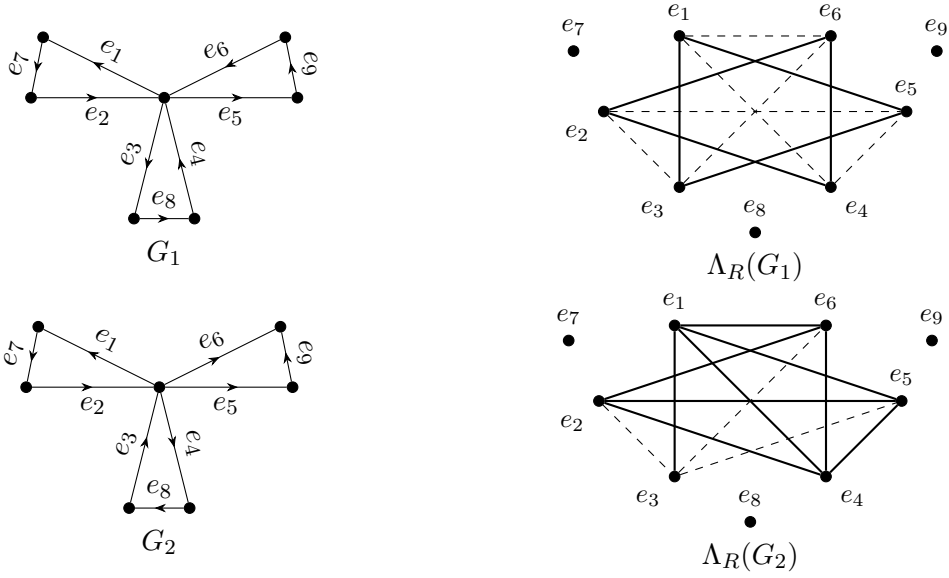
\begin{figure}[]
	\centering
	\begin{minipage}{0.48\linewidth}
	\centerline{
		\begin{tikzpicture}[scale=0.8,vertex1_style/.style={circle,draw,minimum size=0.14cm,inner sep=0pt,fill=black}
			]
			\begin{scope}[xshift=3em]
				\node[vertex1_style, label=below left:\small] (v0) at (0,0) {} ;
				\node[vertex1_style, label=above:\small] (T1) at (-2,1){} ;
				\node[vertex1_style, label=above:\small] (T2) at (-2.2,0) {};
				\draw[decoration={markings, mark = at position 0.6 with {\arrow{Stealth[scale=0.9]}}}, postaction={decorate}] (v0) -- (T1) node[midway,sloped,above] {$e_1$};
				\draw[decoration={markings, mark = at position 0.6 with {\arrow{Stealth[scale=0.9]}}}, postaction={decorate}] (T1) -- (T2) node[midway,sloped,above] {$e_7$};
				\draw[decoration={markings, mark = at position 0.5 with {\arrow{Stealth[scale=0.9]}}}, postaction={decorate}] (T2) -- (v0) node[midway,sloped,below] {$e_2$};				
				\node[vertex1_style, label=above:\small] (T3) at (2,1){} ;
				\node[vertex1_style, label=above:\small] (T4) at (2.2,0) {};
				\draw[decoration={markings, mark = at position 0.5 with {\arrow{Stealth[scale=0.9]}}}, postaction={decorate}] (T3) -- (v0) node[midway,sloped,above] {$e_6$};
				\draw[decoration={markings, mark = at position 0.6 with {\arrow{Stealth[scale=0.9]}}}, postaction={decorate}] (T4) -- (T3) node[midway,sloped,above] {$e_9$};
				\draw[decoration={markings, mark = at position 0.6 with {\arrow{Stealth[scale=0.9]}}}, postaction={decorate}] (v0) -- (T4) node[midway,sloped,below] {$e_5$};				
				\node[vertex1_style, label=above:\small] (T5) at (-0.5,-2){} ;
				\node[vertex1_style, label=above:\small] (T6) at (0.5,-2) {};
				\draw[decoration={markings, mark = at position 0.6 with {\arrow{Stealth[scale=0.9]}}}, postaction={decorate}] (v0) -- (T5) node[midway,sloped,above] {$e_3$};
				\draw[decoration={markings, mark = at position 0.6 with {\arrow{Stealth[scale=0.9]}}}, postaction={decorate}] (T5) -- (T6) node[midway,sloped,above] {$e_8$};
				\draw[decoration={markings, mark = at position 0.5 with {\arrow{Stealth[scale=0.9]}}}, postaction={decorate}] (T6) -- (v0) node[midway,sloped,above] {$e_4$};
			\end{scope}
		\end{tikzpicture}
	}
		\centerline{$G_1$}
	\end{minipage}
\begin{minipage}{0.48\linewidth}
	\centerline{
\begin{tikzpicture}[vertex1_style/.style={circle,draw,minimum size=0.14cm,inner sep=0pt,fill=black}
	]
	\begin{scope}[xshift=3em]
		\node[vertex1_style, label=above :\small$e_1$] (v1) at (-1.0,2.0) {} ;
		\node[vertex1_style, label=below left:\small$e_2$] (v2) at (-2.0,1){} ;
		\node[vertex1_style, label=below left:\small$e_3$] (v3) at (-1.0,0) {};
		\node[vertex1_style, label=below right:\small$e_4$] (v4) at (1.0,0) {} ;
		\node[vertex1_style, label=above:\small$e_5$] (v5) at (2.0,1) {};
		\node[vertex1_style, label=above:\small$e_6$] (v6) at (1.0,2.0) {};
		\node[vertex1_style, label=above:\small$e_7$] (v7) at (-2.4,1.8) {} ;
		\node[vertex1_style, label=above:\small$e_{8}$] (v8) at (0,-.6) {};
		\node[vertex1_style, label=above:\small$e_{9}$] (v9) at (2.4,1.8) {};
		\draw[thick] (v1) -- (v3) ;
		\draw[thick] (v4) -- (v6) ;
		\draw[thick] (v1) -- (v5) ;
		\draw[thick] (v2) -- (v6) ;
		\draw[thick] (v2) -- (v4) ;
		\draw[thick] (v3) -- (v5) ;
		\draw[dashed] (v1) -- (v6) ;
		\draw[dashed] (v1) -- (v4) ;
		\draw[dashed] (v3) -- (v6) ;
		\draw[dashed] (v2) -- (v5) ;
		\draw[dashed] (v2) -- (v3) ;
		\draw[dashed] (v4) -- (v5) ;
	\end{scope}
\end{tikzpicture}
}
\centerline{$\Lambda_R(G_1)$}
\end{minipage}
	\label{fig22}
\begin{minipage}{0.48\linewidth}
		\centerline{
			\begin{tikzpicture}[scale=0.8,vertex1_style/.style={circle,draw,minimum size=0.14cm,inner sep=0pt,fill=black}
				]
				\begin{scope}[xshift=3em]
					\node[vertex1_style, label=below left:\small] (v0) at (0,0) {} ;
					\node[vertex1_style, label=above:\small] (T1) at (-2,1){} ;
					\node[vertex1_style, label=above:\small] (T2) at (-2.2,0) {};
					\draw[decoration={markings, mark = at position 0.6 with {\arrow{Stealth[scale=0.9]}}}, postaction={decorate}] (v0) -- (T1) node[midway,sloped,above] {$e_1$};
					\draw[decoration={markings, mark = at position 0.6 with {\arrow{Stealth[scale=0.9]}}}, postaction={decorate}] (T1) -- (T2) node[midway,sloped,above] {$e_7$};
					\draw[decoration={markings, mark = at position 0.5 with {\arrow{Stealth[scale=0.9]}}}, postaction={decorate}] (T2) -- (v0) node[midway,sloped,below] {$e_2$};					
					\node[vertex1_style, label=above:\small] (T3) at (2,1){} ;
					\node[vertex1_style, label=above:\small] (T4) at (2.2,0) {};
					\draw[decoration={markings, mark = at position 0.5 with {\arrow{Stealth[scale=0.9]}}}, postaction={decorate}] (v0) -- (T3) node[midway,sloped,above] {$e_6$};
					\draw[decoration={markings, mark = at position 0.6 with {\arrow{Stealth[scale=0.9]}}}, postaction={decorate}] (T4) -- (T3) node[midway,sloped,above] {$e_9$};
					\draw[decoration={markings, mark = at position 0.6 with {\arrow{Stealth[scale=0.9]}}}, postaction={decorate}] (v0) -- (T4) node[midway,sloped,below] {$e_5$};					
					\node[vertex1_style, label=above:\small] (T5) at (-0.5,-2){} ;
					\node[vertex1_style, label=above:\small] (T6) at (0.5,-2) {};
					\draw[decoration={markings, mark = at position 0.6 with {\arrow{Stealth[scale=0.9]}}}, postaction={decorate}] (T5) -- (v0) node[midway,sloped,above] {$e_3$};
					\draw[decoration={markings, mark = at position 0.6 with {\arrow{Stealth[scale=0.9]}}}, postaction={decorate}] (T6) -- (T5) node[midway,sloped,above] {$e_8$};
					\draw[decoration={markings, mark = at position 0.5 with {\arrow{Stealth[scale=0.9]}}}, postaction={decorate}] (v0) -- (T6) node[midway,sloped,above] {$e_4$};
				\end{scope}
			\end{tikzpicture}
		}
		\centerline{$G_2$}
	\end{minipage}
	\begin{minipage}{0.48\linewidth}
		\centerline{
			\begin{tikzpicture}[vertex1_style/.style={circle,draw,minimum size=0.14cm,inner sep=0pt,fill=black}
				]
				\begin{scope}[xshift=3em]
					\node[vertex1_style, label=above :\small$e_1$] (v1) at (-1.0,2.0) {} ;
					\node[vertex1_style, label=below left:\small$e_2$] (v2) at (-2.0,1){} ;
					\node[vertex1_style, label=below left:\small$e_3$] (v3) at (-1.0,0) {};
					\node[vertex1_style, label=below right:\small$e_4$] (v4) at (1.0,0) {} ;
					\node[vertex1_style, label=above:\small$e_5$] (v5) at (2.0,1) {};
					\node[vertex1_style, label=above:\small$e_6$] (v6) at (1.0,2.0) {};
					\node[vertex1_style, label=above:\small$e_7$] (v7) at (-2.4,1.8) {} ;
					\node[vertex1_style, label=above:\small$e_{8}$] (v8) at (0,-.6) {};
					\node[vertex1_style, label=above:\small$e_{9}$] (v9) at (2.4,1.8) {};
					\draw[thick] (v1) -- (v3) ;
					\draw[thick] (v4) -- (v6) ;
					\draw[thick] (v1) -- (v5) ;
					\draw[thick] (v2) -- (v6) ;
					\draw[thick] (v2) -- (v4) ;
					\draw[dashed] (v3) -- (v5) ;
					\draw[thick] (v1) -- (v6) ;
					\draw[thick] (v1) -- (v4) ;
					\draw[dashed] (v3) -- (v6) ;
					\draw[thick] (v2) -- (v5) ;
					\draw[dashed] (v2) -- (v3) ;
					\draw[thick] (v4) -- (v5) ;
				\end{scope}
			\end{tikzpicture}
		}
		\centerline{$\Lambda_R(G_2)$}
	\end{minipage}
	\label{222}
	\caption{\small The signed graphs $\Lambda_R(G_1)$ and $\Lambda_R(G_2)$ obtained from $G_1$ and $G_2$ sharing the same underlying graph.}
\end{figure}

The following result follows from the  Perron-Frobenius Theory on irreducible matrices and the previous propositions.

\begin{prop}\label{lem-5}
Let $G$ be a graph with no induced odd cycle with length greater than $3$. Then there is an
orientation such that $\mathcal{H}(G)$ has a non-negative eigenvector $\mathbf{x}$
corresponding to the largest eigenvalue. Moreover, $\mathbf{x}$ is positive if
$\Lambda_R(G)$ is connected.
\end{prop}

A graph $G$ is called {\it $k$-triangle-regular} if
$\triangle(e)=k$ holds for every edge $e \in E(G)$. Therefore, in this case, the eigenvalues of $\mathcal{H}(G)$
are those of $\Lambda_R(G)$ added $k+2$. This enables us to give more examples in which we compute the $\mathcal{H}$-spectrum of a particular graph.

\begin{examplex}We know from \cite{SiSt} that the spectrum of a negative signed cycle of odd length $n$ contains the numbers $2\cos \frac{\pi j}{n}$, $j\in\{1, 3, \ldots, n-1\}$, all with multiplicity 2, and $-2$ as an additional eigenvalue. Therefore, for $n\geq 3$, by employing \eqref{eq:rel} we obtain
	$$
	{\Sp}_{\mathcal{H}}(C_n) =
	\left\{\!\!\begin{array}{cccccc}
		2\cos\frac{\pi}{n}+2  & 2\cos\frac{3\pi}{n}+2 & \cdots & 2\cos\frac{(n-1)\pi}{n}+2& 0
		\\
		2 & 2 & \cdots & 2&1
	\end{array}
	\!\!\right\}.
	$$
Of course, the $\mathcal{H}$-spectrum of a triangle is computed as in Example~\ref{ex-1}, whereas the spectrum of any even cycle is computed on the basis of the spectrum its adjacency matrix (to be found in \cite[p.~53]{cve-doob-sachs-book}) and Proposition~\ref{prop:I}.
\end{examplex}

For the graphs $G$ and $H$, $G\vee H$ denotes their \textit{join}, i.e., the graph obtained by inserting an edge between every vertex of $G$ and every vertex of $H$.

\begin{examplex}\label{exa:friend} The \textit{friendship graph} is the graph $K_1\vee nK_2$, i.e., it is  the cone over $n$ disjoint edges. Now, $\Lambda_R(K_1\vee nK_2)$ switches to its underlying graph for any edge orientation on $K_1\vee nK_2$, by Proposition~\ref{prop:I}. Therefore, we may assume that the edges of $\Lambda_R(K_1\vee nK_2)$ are positive. In this case, $\Lambda_R(K_1\vee nK_2)$ consists of the complementary graph of $nK_2$ (the so-called  cocktail party graph on $2n$ vertices) and the set of $n$ isolated vertices. Its spectrum is comprised of $2n-2$, 0 with multiplicity $2n$ and $-2$ with multiplicity $n-1$. Since the friendship graph is $1$-triangle-regular, from \eqref{eq:rel} we obtain	
	$$
	{\Sp}_{\mathcal{H}}(K_1\vee nK_2) =
	\left\{\!\!\begin{array}{ccc}
		2n+1 & 3 & 1
		\\
		1  & 2n  & n-1
	\end{array}
	\!\!\right\}.
	$$
\end{examplex}

\begin{examplex}\label{exa:icosa}The \textit{icosahedron graph} $G$ is one of the 5 Platonic solids. It has 12 vertices and 30 edges, and every edge is contained in exactly 2 triangles. It has no induced odd cycles of length 5 or more, and therefore $\Lambda_R(G)$ is an all-positive signed graph obtained from the line graph of $G$ by deleting the edges that correspond to triangles in~$G$. It $\mathcal{H}$-spectrum is computed easily and reads	
$$
{\Sp}_{\mathcal{H}}(G) =
\left\{\!\!\begin{array}{cccccc}
	8 & 6 & 4&3&2&1
	\\
	1  & 11  & 5&4&5&4
\end{array}
\!\!\right\}.
$$	
\end{examplex}

It is not easy to determine whether a signed graph, say $\dot{H}$, corresponds to a graph $G$ in the sense that $\dot{H}\cong \Lambda_R(G)$, which brings us to the following problem.

\begin{problem}
Find a sufficient and a necessary condition for a signed graph to be $\Lambda_R(G)$ for some graph $G$.
\end{problem}


\section{Interlacing Theorem for $\mathcal{H}$-eigenvalues}\label{H42}

 For a
Hermitian matrix $M$ of order $n$, let
\begin{equation*}\label{her-eigen}
\theta_1(M)\geq\theta_2(M)\geq\cdots\geq\theta_n(M),
\end{equation*}
denote its eigenvalues. The following result is well-known.

\begin{thm}[{\cite[Interlacing Theorem]{god-roy}}]\label{cau-inter}
Let $M$ be a Hermitian matrix of order $m$, and let $N$ be its principal submatrix of order
$n$ $(1 \leq n \leq m)$. For every integer $1 \leq i \leq n$, it holds $\theta_{m-n+i}(M) \leq
\theta_i(N) \leq \theta_i(M)$.
\end{thm}

For a subset $S\subseteq V(G)$, the {\it induced subgraph} $G[S]$ is the graph whose  vertex set is $S$ and whose edge set consists of all edges in $E(G)$ with both endpoints in $S$. In other words, $G[S]$ is obtained from $G$ by deleting all vertices not in $S$ together with their incident edges. Note that the adjacency matrix $A(G[S])$ of an induced subgraph $G[S]$ must be a principal submatrix of the adjacency matrix $A(G)$ of $G$. This establishes the famous Interlacing Theorem as one of the most powerful tool to investigate  spectrum of the adjacency matrix. Despite the Helmholtzian $\mathcal{H}(G[S])$ of $G[S]$ may not be a principal submatrix of $\mathcal{H}(G)$, some  results can be offered by applying the Interlacing Theorem.

For an induced subgraph $G'$ of $G$ and an edge $e\in E(G')$, denote by $\kappa_{G'}(e)=\triangle_G(e)-\triangle_{G'}(e)$, $\kappa_{\min}(G')=\min\{\kappa_{G'}(e)\mid e\in E(G')\}$ and $\kappa_{\max}(G')=\max\{\kappa_{G'}(e)\mid e\in E(G')\}$.

\begin{thm}\label{thm-4-8}
Let $G$ be a graph with $\mathcal{H}$-eigenvalues $\lambda_1\ge\lambda_2\ge \cdots \ge \lambda_m$. If $G'$ is an induced subgraph of $G$ with $\mathcal{H}$-eigenvalues $\lambda_1'\ge\lambda_2'\ge\cdots\ge\lambda_{m'}'$, then $\lambda_{i}\ge\lambda_i'+\kappa_{\min}(G')$ and $\lambda_{i'}+\kappa_{\max}(G')\ge \lambda_{m-m'+i}$ for $1\le i\le m'$.
\end{thm}

\begin{proof}
Keep the orientations of edges in $E(G')$ to coincide with those in $E(G)$. Accordingly, $\mathcal{H}'=\mathcal{H}(G')+\diag(\kappa_{G'}(e)\mid e\in E(G'))$ is a principal submatrix of $\mathcal{H}(G)$. By Theorem \ref{cau-inter}, we have $\lambda_{i}\ge \lambda_i(\mathcal{H}')\ge\lambda_{m-m'+i}$ for $1\le i\le m'$, and the result follows from the well-known Courant-Weyl Inequalities \cite[Theorem 2.8.1]{bro-hae-book}.
\end{proof}

As a simple application, we get the following result.
\begin{cor}
Let $G$ be a graph with $\mathcal{H}$-eigenvalues $\lambda_1\ge\lambda_2\ge\cdots\ge\lambda_m$. If $G$ contains a clique $K_t$, then $\lambda_{t\choose 2}\ge t$. In particular, if the clique number of $G$ is $\omega$, then $\lambda_{\omega \choose 2}\ge \omega$.
\end{cor}
\begin{proof}
The result follows from Theorem \ref{thm-4-8} and Example \ref{ex-1}.
\end{proof}

We may single out a special case of the previous theorem.

\begin{cor}\label{cor-4-1}
If $G'$ is an induced subgraph of $G$ with $\kappa_{\min}(G')=\kappa_{\max}(G')=\kappa$, then, for $1\le i\le m'$,
\[\lambda_i(G)\ge\lambda_i(G')+\kappa\ge\lambda_{m-m'+i}(G),\]
where $m=|E(G)|$ and $m'=|E(G')|$.
\end{cor}

A {\it block} of $G$ is a maximal induced subgraph without cut vertices (i.e., vertices whose removal increases the number of connected components), and if $G'$ is formed by some blocks of $G$ then $\kappa_{\max}(G')=0$. This leads to the following result.

\begin{cor}
If $G'$ is an induced subgraph of $G$ formed by some blocks of $G$, then $\lambda_i(G)\ge \lambda_i(G')\ge\lambda_{m-m'+i}(G)$ for $1\le i\le m'$, where $m=|E(G)|$ and $m'=|E(G')|$.
\end{cor}

For example, in Fig.~\ref{fig-1}, the subgraph $G'$ of $G$ induced by $\{2,3,4,5\}$ is a block. We have $\lambda_1'=\lambda_2'=\lambda_3'=4$ and $\lambda_4'=\lambda_5'=2$, and thus $\lambda_1\ge 4$, $\lambda_2=\lambda_3=4$, $4\ge\lambda_4\ge 2$, $\lambda_5=2$ and $2\ge\lambda_6$, where $\lambda_i$ and $\lambda_i'$ are $\mathcal{H}$-eigenvalues of $G$ and $G'$, respectively. In fact, $\lambda_1=5$, $\lambda_2=\lambda_3=4$, $\lambda_4=\lambda_5=2$ and $\lambda_6=1$.

Note that a longest path in $G$ must be an induced path. The following result follows.
\begin{cor}
Let $G$ be a connected graph with size $m$ and diameter $d$. If $P_{d+1}$ is a longest path in $G$, then $\lambda_{m-d+1}\le 2+2\cos{\pi/d}+\kappa_{\max}(P_{d+1})$. In particular, if there is a path $P_{d+1}$ in $G$ with $\kappa_{\max}(P_{d+1})=0$, then $\lambda_{m-d+1}\le 2+2\cos{\pi/d}$.
\end{cor}

A \textit{twin} (resp.~\textit{co-twin}) of a vertex $v$ is  its non-neighbour (neighbour) sharing the same neighbourhood with~$v$. We consider particular induced subgraphs obtained by deleting twin or co-twin vertices.

\begin{prop}
	Let $v$ be a vertex of a graph $G$, and $e_1, e_2, \ldots, e_k$ be edges incident with $v$. Let $H$ be a graph obtained from $G$ by adding a twin (resp.~co-twin) to $v$. Then the $\mathcal{H}$-spectrum of $H$ contains $\triangle_G(e_1)+2, \triangle_G(e_2)+1, \ldots, \triangle_G(e_k)+1$ (resp.~$\triangle_G(e_1)+3, \triangle_G(e_2)+3, \ldots, \triangle_G(e_k)+3$).
\end{prop}

\begin{proof} Let $v'$ be a twin (resp.~co-twin) of $v$, and let $e_1', e_2', \ldots, e_k'$ be the corresponding edges. Without loss of generality, we may assume that the edge $e_i$ and its twin $e_i'$ are oriented in the same way, say the first to the vertex $v$, and the second to the vertex $v'$. Let $\mathbf{x}$ be the vector taking $1$ at $e_i$, $-1$ at its twin $e_1'$ and 0 at all remaining edges of $H$. We have $\mathcal{H}(H)\mathbf{x}=(\triangle_G(e_i)+1)\mathbf{x}$ if $v'$ is a twin, and $\mathcal{H}(H)\mathbf{x}=(\triangle_G(e_i)+3)\mathbf{x}$ if $v'$ is a co-twin. The latter follows since there is an additional triangle containing both $e_i$ and $e_i'$. The proof is completed.
\end{proof}

\section{Relationships between the $\mathcal{H}$-spectrum and the Laplacian spectrum}\label{H41}

We still assume that $G$ has $n$ vertices, $m$ edges and $t$ triangles. Recall that its  Helmholtzian matrix can be written as $\mathcal{H}(G)=\mathcal{B}\mathcal{B}^\intercal+\mathcal{C}^\intercal\mathcal{C}=\big(\mathcal{B}~~ \mathcal{C}^\intercal\big)\big(\mathcal{B}~~ \mathcal{C}^\intercal\big)^\intercal$, where $\mathcal{B}=(b_{ev})_{m\times n}$ and $\mathcal{C}=(c_{\vartriangle e})_{t\times m}$ are the edge-vertex and the triangle-edge incidence matrices, respectively. Let $\mathcal{H}'(G)=\big(\mathcal{B}~~ \mathcal{C}^\intercal\big)^\intercal \big(\mathcal{B} ~~\mathcal{C}^\intercal\big)$. Since $\mathcal{H}(G)$ and $\mathcal{H}'(G)$ share the same non-zero eigenvalues,  we investigate the eigenvalues of the latter matrix, which is given by
\[\mathcal{H}'(G)=\big(\mathcal{B}~~ \mathcal{C}^\intercal\big)^\intercal \big(\mathcal{B}~~ \mathcal{C}^\intercal\big)=\begin{pmatrix}\mathcal{B}^\intercal\mathcal{B}&\mathcal{B}^\intercal\mathcal{C}^\intercal\\ \mathcal{C}\mathcal{B}&\mathcal{C}\mathcal{C}^\intercal\end{pmatrix}.\]
Clearly,  $\mathcal{B}^\intercal\mathcal{B}$ is just the Laplacian matrix $L(G)$.

By a direct computation, for any vertex $v\in V(G)$ and any triangle $\vartriangle\in T(G)$, we obtain
\[(\mathcal{C}\mathcal{B})_{\vartriangle v}=\sum_{e\in E(G)}c_{\vartriangle e} b_{ev}.\]
If  $c_{\triangle e} b_{ev}\ne 0$, then $c_{\vartriangle e}\ne 0$ and $b_{ev}\ne 0$. This means that $v\in e$ and $e\in \vartriangle$, and thus $v\in \vartriangle$. Therefore, we have $c_{\vartriangle e} b_{ev}\ne 0$ only if $v\in \vartriangle$. In this case, there are two edges $e_1,e_2\in \vartriangle$ with $\{e_1^+,e_1^-\}\cap \{e_2^+,e_2^-\}=\{v\}$, and $(\mathcal{C}\mathcal{B})_{\vartriangle,v}=c_{\vartriangle e_1} b_{e_1v}+c_{\vartriangle e_2} b_{e_2v}$. If $e_1\rightarrow v\rightarrow e_2$ or $e_2\rightarrow v\rightarrow e_1$, by  definitions of $\mathcal{B}$ and $\mathcal{C}$, we have $b_{e_1v}+b_{e_2v}=0$ and $c_{\vartriangle e_1}=c_{\vartriangle e_2}$. This leads to $(\mathcal{C}\mathcal{B})_{\vartriangle v}=0$. If $v\rightarrow e_1,e_2$ or $e_1,e_2\rightarrow v$, then $b_{e_1 v}=b_{e_2v}$ and $c_{\vartriangle e_1}+c_{\vartriangle e_2}=0$, which again leads to $(\mathcal{C}\mathcal{B})_{\vartriangle v}=0$. Hence, $\mathcal{C}\mathcal{B}=O$.

It remains to consider the matrix $\mathcal{C}\mathcal{C}^\intercal$. For  $\vartriangle_1,\vartriangle_2\in T(G)$, we have
\[(\mathcal{C}\mathcal{C}^\intercal)_{\vartriangle_1\vartriangle_2}=\sum_{e\in E(G)}c_{\vartriangle_1 e}c_{\vartriangle_2 e}.\]
If $c_{\vartriangle_1 e}c_{\vartriangle_2 e}\ne 0$, then $e\in \vartriangle_1$ and $e\in \vartriangle_2$. If $\vartriangle_1=\vartriangle_2$, then there are three edges $e_1,e_2,e_3\in\vartriangle_1$, and thus $(\mathcal{C}\mathcal{C}^\intercal)_{\vartriangle_1\vartriangle_2}=3$. If $\vartriangle_1\ne\vartriangle_2$, then $c_{\vartriangle_1 e}c_{\vartriangle_2 e}\ne 0$ holds only if there is an edge $e$ belonging to both $\vartriangle_1$ and $\vartriangle_2$. In this case, we have $(\mathcal{C}\mathcal{C}^\intercal)_{\vartriangle_1\vartriangle_2}=c_{\vartriangle_1 e}c_{\vartriangle_2 e}$. If $e\in \vartriangle_i^+$ for $i\in\{1,2\}$, or $e\in \vartriangle_i^-$ for $i\in\{1,2\}$, then $c_{\vartriangle_1 e}=c_{\vartriangle_2 e}$, and thus $(\mathcal{C}\mathcal{C}^\intercal)_{\vartriangle_1\vartriangle_2}=1$; otherwise, $(\mathcal{C}\mathcal{C}^\intercal)_{\vartriangle_1\vartriangle_2}=-1$.

To create a more clear picture about the matrix $\mathcal{C}\mathcal{C}^\intercal$, we introduce a {\it triangular signed graph} $G_{\vartriangle}$ of $G$. The vertex set $V(G_{\vartriangle})$ coincides with $T(G)$, and two triangles $\vartriangle_1$ and $\vartriangle_2$ are joined by a positive edge  if they share a common edge $e$ with either $e\in \vartriangle_i^+$ for $i\in\{1,2\}$, or $e\in \vartriangle_i^-$ for $i\in\{1,2\}$. They are joined by a negative edge if they share a common edge $e$ with either $e\in\vartriangle_1^+$ and $e\in\vartriangle_2^-$ or $e\in\vartriangle_1^-$ and $e\in\vartriangle_2^+$. Denote by $A(G_{\vartriangle})$ the adjacency matrix of $G_{\vartriangle}$. It is clear that $G_{\vartriangle}$  depends on the orientation of  triangles, but not on the orientation of  edges.

The previous discussion leads to
\begin{equation}\label{eq:H'}\mathcal{H}'(G)=\begin{pmatrix}L(G)&0\\ 0&3I+A(G_{\vartriangle})\end{pmatrix}.\end{equation}
Thus, the following result follows.

\begin{thm}\label{4-9-thm1}
	Let $G$ be a graph with Laplacian eigenvalues $\mu_1\ge\mu_2\ge\cdots\ge\mu_n$, and let $\eta_1\ge \eta_2\ge\cdots\ge\eta_t$ be the eigenvalues of $A(G_{\vartriangle})$ of \eqref{eq:H'}. Then the non-zero eigenvalues of $\mathcal{H}(G)$ are all $\mu_i$ with $\mu_i\ne 0$ and all $3+\eta_j$ with $\eta_j\ne -3$. In particular, the largest $\mathcal{H}$-eigenvalue is $\max\{\mu_1,3+\eta_1\}$.
\end{thm}

From Theorem \ref{4-9-thm1}, to investigate the  $\mathcal{H}$-spectrum of $G$, it suffices to investigate the spectrum of $L(G)$ and the spectrum of $A(G_{\vartriangle})$. Since $L(G)$ is well studied, in what follows we focus our attention to $G_{\vartriangle}$. In fact, the triangular signed graph $G_{\vartriangle}$ reveals  relations between triangles of $G$. From Theorem~\ref{lem-2}, we get the following simple result.

\begin{cor}
	The spectrum of $G_{\vartriangle}$ does not depend on the orientation of $G$.
\end{cor}

For a triangle $\vartriangle\in T(G)$, let $e_1,e_2,e_3$ be its edges. Therefore, the degree of $\vartriangle$ in $G_{\vartriangle}$ is $d(\vartriangle)=\triangle(e_1)+\triangle(e_2)+\triangle(e_3)-3\le 3\max\triangle(e)-3$. This leads to the following result.

\begin{cor}\label{cor:l1m1}
	For the largest $\mathcal{H}$-eigenvalue  $\lambda_1$ of $G$, $\lambda_1\le\max\{\mu_1,3\max\triangle(e)\}$.
\end{cor}

Note that $\mu_1\ge\Delta(G)+1$. If for the number of triangles $t(G)\le \lfloor (\Delta(G)+1)/3\rfloor$, then $\mu_1\ge3\max\triangle(e)$, and thus $\lambda_1=\mu_1$.

\begin{cor}
	If $t(G)\le \lfloor (\Delta(G)+1)/3\rfloor$, then $\lambda_1=\mu_1$.
\end{cor}

By inspecting graphs with at most six vertices, we did not find any with $\lambda_1\neq \mu_1$. Accordingly, we pose the following problem.

\begin{problem}
	Prove or disprove that $\lambda_1=\mu_1$ holds for every graph, where $\lambda_1$ and $\mu_1$ are its largest $\mathcal{H}$-eigenvalue and its largest Laplacian eigenvalue, respectively.
\end{problem}

To close this part, we present a relation between $\mathcal{H}$-eigenvalues and the cliques of $G$. We need the Caro-Wei's Theorem.

\begin{lem}[Caro \cite{Caro}; Wei \cite{Wei}]\label{4-9-lem2}
	For a graph $G$  with independence number $\alpha(G)$, $\alpha(G)\ge\sum_{v\in V(G)}\frac{1}{d(v)+1}$.
\end{lem}

Note that, if $G\cong K_n$, then every vertex $\vartriangle\in V(G_{\vartriangle})$ has degree $3(n-3)$ in $G_{\vartriangle}$. The following result follows immediately from Lemma \ref{4-9-lem2}.

\begin{lem}\label{4-9-lem2a}
	If $G\cong K_n$, then $\alpha(G_{\vartriangle})\ge\frac{{n\choose 3}}{1+3(n-3)}$.
\end{lem}
\begin{proof}
For each vertex $\vartriangle\in G_{\vartriangle}$, its degree in $G_{\vartriangle}$ satisfies $d(\vartriangle)=3(n-3)$. Thus, Lemma \ref{4-9-lem2} implies that $\alpha(G_{\vartriangle})\ge \sum_{\vartriangle \in V(G_{\vartriangle})}\frac{1}{3(n-3)+1}=\frac{\binom{n}{3}}{3(n-3)+1}$.
\end{proof}

If $K_a$ and $K_b$ are edge-disjoint cliques in $G$, then the triangles in $K_a$ and those in $K_b$ do not share any edge. Therefore, the corresponding vertices in $G_{\vartriangle}$ are not adjacent. From Lemma \ref{4-9-lem2a}, we obtain the following one.

\begin{lem}\label{4-9-lem3}
	If $K_{s_1}, K_{s_2},\ldots,K_{s_t}$ are edge-disjoint cliques of $G$ with $s_i\ge 3$, then the independence number of $G_{\vartriangle}$ satisfies
	\[\alpha(G_{\vartriangle})\ge \sum_{i=1}^t\frac{{s_i\choose 3}}{1+3(s_i-3)}.\]
\end{lem}
\begin{proof}
Assume $G^{(i)}=K_{s_i}$ for $1\le i\le t$. Since $G^{(1)},\ldots, G^{(t)}$ are edge-disjoint, the graphs $G^{(i)}_{\vartriangle}$ are vertex-disjoint induced subgraphs and there is no edge between $G^{(i)}_{\vartriangle}$ and $G^{(j)}_{\vartriangle}$ for $i\ne j$. Therefore, from Lemma \ref{4-9-lem2a}, we get
\[\alpha(G_{\vartriangle})\ge \sum_{i=1}^t\alpha(G^{(i)}_{\vartriangle})\ge\sum_{i=1}^t\frac{\binom{s_i}{3}}{1+3(s_i-3)}.\]
\end{proof}

Combining Theorem \ref{cau-inter} and Lemma \ref{4-9-lem3}, we get the following result.

\begin{thm}\label{4-9-thm2}
	Let $K_{s_1}, K_{s_2},\ldots,K_{s_t}$ be edge-disjoint cliques of a graph $G$ with $s_i\ge 3$. Then $G$ has at least $\sum_{i=1}^t\frac{{s_i\choose 3}}{1+3(s_i-3)}$ $\mathcal{H}$-eigenvalues not less than $3$.
\end{thm}
\begin{proof}
Assume that $H$ is the subgraph induced by the maximum independent set of $G_{\vartriangle}$. It is clear that $\lambda_1(A(H))=\cdots=\lambda_{|\alpha(G_{\vartriangle})|}=0$. According to Lemma \ref{cau-inter}, we get $\lambda_{\alpha(G_{\vartriangle})}(A(G_{\vartriangle}))\ge 0$. Therefore, $G$ has at least $\alpha(G_{\vartriangle})$ $\mathcal{H}$-eigenvalues not less than $3$.  Hence, the result follows from Lemma \ref{4-9-lem3}.
\end{proof}

In particular, we get the following result.
\begin{cor}\label{4-9-cor1}
	If $\omega$ is the clique number of  a graph $G$, then $G$ has at least $\frac{{\omega\choose 3}}{1+3(\omega-3)}$ $\mathcal{H}$-eigenvalues not less than $3$.
\end{cor}

\section{Remarks}\label{sec:rem}

In the paper, we prove the $\mathcal{H}$-eigenvalues of graphs to be non-negative, and consider the irreducibility and eigenvalue interlacing. We also emphasize a relation between the $\mathcal{H}$-spectrum of a graph and the spectra of Laplacian matrix of the same graph and  adjacency matrix of the corresponding triangular graph, given in \eqref{eq:H'}. Some research problems faced during the investigation are formulated. Inspired by Exam. \ref{ex-1}, another paper about the graphs with few distinct $\mathcal{H}$-eigenvalues is forthcoming.

Although we have considered only finite, unoriented, simple graphs,  we relate them to graph generalizations including  oriented graphs, weighted graphs and signed graphs.  Similarly, we have considered the spectrum of the Helmholtzian matrix, but we also discover an interplay between it and  the adjacency spectra or the Laplacian spectra of certain related  structures. Accordingly,  the $\mathcal{H}$-spectrum bridges and  balances between several graph types and their structural and spectral relationships.

This paper is a building block, and we expect more come.

\section*{Declaration of Interest Statement}

The authors declare that they have no known competing financial interests or personal relationships that could
have appeared to influence the work reported in this paper.

\section*{Data availability statement}

There is no associated data.

\section*{Acknowledgements}

Jianfeng Wang  would like to thank Professor  L.-H. Lim for his helpful suggestions, and Shu Li as well as Zhen Chen for their discussions.

Lu Lu is supported by the National Natural Science Foundation of China (No.~12371362). Yongtang Shi is supported by the National Natural Science Foundation of China (No.~11922112). Zoran Stani\'c is supported by the Ministry
of Science, Technological Development, and Innovation of the Republic of Serbia (No.~451-03-136/2025-
03/200104). Jianfeng Wang is supported by the National Natural Science Foundation of China (No.~12371353).  Yi Wang is supported by the National Natural Science Foundation of China (No.~12171002 and 12331012).
%
%

{}

\addcontentsline{toc}{section}{References}

\begin{thebibliography}{99}

\bibitem{BSZ}
F. Belardo, Z. Stani\'c, T. Zaslavsky, Total graph of a signed graph, Ars Math. Contemp., 23 (2023), $\#$P1.02.

\bibitem{bro-hae-book}
A.E. Brouwer, W.H. Haemers, Spectra of Graphs,  Springer, Berlin 2012.

\bibitem{Caro}
Y. Caro, New Results on the Independence Number, Technical report, Tel Aviv University, 1979.

\bibitem{chung-book}
F. Chung, L.Y. Lu,  Complex Graphs and Networks, The American Mathematical Society, New York, 2006.

\bibitem{cve-doob-sachs-book}
D.M. Cvetkovi\'{c}, M. Doob, H. Sachs, Spectra of Graphs, third edition, Johann Ambrosius
Barth, Heidelberg-Leipzig, 1995.

\bibitem{god-roy}
C. Godsil, G. Royle,  Algebraic Graph Theory, Springer, Berlin, 2001.

\bibitem{jiang-lim}
X. Jiang, L.-H. Lim, Y. Yao, Y. Ye, Statistical ranking and combinatorial Hodge theory, Math. Program. 127 (2011) 203--244.

\bibitem{le}
V.B. Le, Gallai graphs and anti-Gallai graphs, Discrete Math. 159 (1996) 179--189.

\bibitem{li-lu-wang}
S. Li, L. Lu, J.F. Wang, A graph discretization of vector Laplacian, Discrete Appl. Math.
379 (2026) 446--460.

\bibitem{Hodge-lim}
L.-H. Lim, Hodge Laplacians on graphs, SIAM Rev. 62 (2020) 685--715.

\bibitem{sch-Siam}
M.T. Schaub, A.R. Benson, P. Horn, G. Lippner,  A. Jadbabaie, Random Walks on Simplicial Complexes and the Normalized Hodge 1-Laplacian, SIAM Rev.  62 (2020) 353--391.

\bibitem{shi-chen}
D.H. Shi, G.R. Chen, Simplicial networks: a powerful tool for characterizing higher-order
interactions, Nat. Sci. Rev.  9(5) (2022) nwac038.

\bibitem{SiSt}
S.K. Simi\' c, Z. Stani\' c, Polynomial reconstruction of signed graphs, Linear Algebra Appl. 501 (2016) 390--408.

\bibitem{ifge}
Z. Stani\' c, Inequalities for Graph Eigenvalues, Cambridge University Press, Cambridge, 2015.

\bibitem{wat-str}
D.J. Watts, S.H. Strogatz, Collective dynamics of 'small world' networks, Nature 393 (1998) 440--442.

\bibitem{Wei}
V. Wei, A lower bound on the stability number of a simple graph, Technical report, Bell
Laboratories Technical Memorandum, 1981.

\bibitem{wis-dam}
P. Wissing, E.R. van Dam, Spectral fundamentals and characterizations of signed directed graphs, J. Comb. Theory, Ser. A 187 (2022) 105573

\bibitem{Zas}
T. Zaslavsky, Matrices in the theory of signed simple graphs, in: B.D. Acharya, G.O.H. Katona, J. Ne\v set\v ril (Eds.), Advances in Discrete Mathematics and Applications: Mysore 2008, Ramanujan Math. Soc., Mysore, 2010, pp.~207--229.
\end{thebibliography}
\end{document}